\documentclass[sn-mathphys]{sn-jnl}

\jyear{2023}%

\usepackage{graphicx}
\usepackage{caption}
\usepackage{subcaption}
\providecommand{\norm}[1]{\lVert#1\rVert}
\usepackage{color}
\usepackage{amssymb}
\usepackage{amsmath}
\DeclareMathOperator*{\argmin}{arg\,min}
\theoremstyle{thmstyleone}%
\newtheorem{theorem}{Theorem}
\theoremstyle{thmstyletwo}%
\newtheorem{example}{Example}%
\newtheorem{remark}{Remark}%
\newtheorem{corollary}{Corollary}

\theoremstyle{thmstylethree}%
\newtheorem{definition}{Definition}%

\newcommand{\red}{\textcolor{black}}

\begin{document}

\title[Proximal gradient methods with inexact oracle]{Proximal gradient methods  with inexact oracle of degree $q$ for composite optimization}

\author[1]{\fnm{Yassine} \sur{Nabou}}\email{yassine.nabou@stud.acs.upb.ro}

\author[2,3]{\fnm{Fran\c cois} \sur{Glineur}}\email{francois.glineur@uclouvain.be}

\author[1,4]{\fnm{Ion} \sur{Necoara}}\email{ion.necoara@upb.ro}

\affil[1]{\orgdiv{Automatic Control and Systems Engineering Department}, \orgname{National University of Science and Technology Politehnica Bucharest}, \orgaddress{\street{Spl. Independentei, 313}, \city{Bucharest}, \postcode{060042}, \country{Romania}}}

\affil[2]{\orgdiv{ICTEAM Institute}, \orgname{Universit\'{e} catholique de Louvain}, \orgaddress{\city{Louvain-La-Neuve}, \postcode{1348}, \country{Belgium}}}

\affil[3]{\orgdiv{CORE}, \orgname{Universit\'{e} catholique de Louvain}, \orgaddress{\postcode{1348}, \country{Belgium}}}

\affil[4]{\orgdiv{Gheorghe Mihoc-Caius Iacob} \orgname{Institute of Mathematical Statistics and Applied Mathematics of the Romanian Academy}, \city{Bucharest}, \postcode{610101}, \country{Romania}}


\abstract{We introduce the concept of inexact first-order oracle of degree $q$ for a possibly nonconvex and nonsmooth function, which naturally appears in the context of approximate gradient, weak level of smoothness and other situations. Our definition is less conservative than those found in the existing literature, and \red{ it can be viewed as an interpolation between fully exact and the existing inexact first-order oracle definitions}. We analyze the convergence behavior of a (fast) inexact proximal gradient method using such an oracle for solving (non)convex composite minimization problems.  We derive complexity estimates and study the dependence between the accuracy of the oracle and the desired accuracy of the gradient or of the objective function. Our results show that better rates can be obtained both theoretically and in numerical simulations  when $q$ is large.}



\maketitle

\section{Introduction}
\label{intro}
Optimization methods based on gradient information are widely used in applications where high accuracy is not desired, such as machine learning, data analysis, signal processing and statistics \cite{Pol:63,Bec:17,Bot:10,WaLI:14}. The standard convergence analysis of gradient-based methods requires the availability of the exact gradient information for the objective function. However, in many optimization problems, one doesn't have access to exact gradients, e.g., the gradient is obtained by solving another optimization problem. In this case  one can use  inexact (approximate) gradient information. In this paper, we consider the following composite optimization problem:
\begin{align}\label{eq:opt_prblm}
\min\limits_{x\in\mathbb{E}} f(x): = F(x) + h(x),
\end{align}  
where $h: \mathbb{E} \to \bar{\mathbb{R}}$ is a simple (i.e., proximal easy) closed convex function, $F: \mathbb{E} \to  \mathbb{R}$ is a general lower semicontinuous function (possibly nonconvex) and there exist $f_\infty$ such that $f(x)\geq f_\infty \red{>-\infty}$ for all  $x \in\text{dom} \,f = \text{dom}\, h$. We assume that we can compute exactly the proximal operator of $h$, and that we cannot have access to the (sub)differential of $F$, but can compute an approximation of it at any given point.  Optimization algorithms with inexact first-order oracles are well studied in literature, see e.g.,  \cite{BoGuRa:16,CoDiOr:18,Asp:08,DeGlNe:13,DvGa:16,Dvur:17,StDv:19}.  For example, \cite{DeGlNe:13} considers the case where $h$ is the indicator function of \red{a convex set $Q$} and $F$ is a  convex function, and \red{introduces} the so-called inexact first-order \red{$(\delta,L)$-oracle} for $F$, i.e., for any $y\in Q$ one can compute an inexact oracle consisting of a pair $(F_{\delta,L}(y), g_{\delta, L}(y))$ such that: 
\begin{align}\label{def:inxt}
0\leq F(x) - \Big(F_{\delta,L }(y) + \langle g_{\delta,L}(y),x-y \rangle\Big)\leq \frac{L}{2}\norm{x-y}^{2} + \delta \quad \forall x \in Q.
\end{align}
Then,  \cite{DeGlNe:13} introduces (fast) inexact first-order methods based on $g_{\delta, L}(y)$ information and  derives asymptotic convergence in function values of order $\mathcal{O}\left(\frac{1}{k} + \delta \right)$ or $\mathcal{O}\left(\frac{1}{k^2} + k\delta\right)$, respectively. One can notice that in the nonaccelerated scheme, the objective function accuracy decreases with $k$ and asymptotically tends to $\delta$, while in the {accelerated scheme}, there is error accumulation. Further, \cite{Dvur:17} considers problem \eqref{eq:opt_prblm} with domain of $h$  bounded, and introduces the following inexact first-order oracle:
\begin{align*}
\vert F(x) - F_{\delta,L}(x)\vert \leq \delta,\;\;
F(x) - F_{\delta,L}(y) - \langle g_{\delta,L}(y),x-y \rangle\leq \frac{L}{2}\| x - y \|^2 + \delta.
\end{align*} 
Under the assumptions that  $F$ is nonconvex and $h$ is convex but with bounded domain,  \cite{Dvur:17} derives a sublinear rate in the squared norm of the generalized gradient mapping of order $\mathcal{O}\left(\frac{1}{k} + \delta\right) $ for an inexact proximal gradient method based on $g_{\delta, L}(y)$ information. Note that all previous results provide convergence rates under the assumption of the boundedness of the domain of $f$ (or equivalently of $h$). An open question is whether one can modify the previous definitions of inexact first-order oracle to cover both the convex and nonconvex settings, in order to be more general and to improve the convergence results of an algorithm based on this inexact information. \red{More precisely, can one define a general inexact oracle that bridges the gap between exact oracle (exact gradient information) and the existing inexact first-order oracle definitions found in the literature \cite{DeGlNe:13, Dvur:17}?} In this paper we answer this question positively for both convex and nonconvex problems, introducing a suitable definition of inexactness for a first-order oracle for $F$ involving some degree $0\leq q < 2$, which consist in multiplying the constant $\delta$ in \eqref{def:inxt} with quantity $\| x - y \|^q$ (see  Definition \ref{eq:def1}). We provide several examples that can fit in our proposed inexact first-order oracle framework, such as approximate gradient or weak level of smoothness, and show that, under this new definition of inexactness, we can remove the boundedness assumption of the domain of $h$. Then, we consider an inexact proximal gradient algorithm  based on this inexact first-order oracle and provide convergence rates of order $\mathcal{O}\left(\frac{1}{k} + \delta^{2/(2-q)} \right)$ for  $q \in [0, 1)$ and $\mathcal{O}\left(\frac{1}{k} + \frac{\delta}{k^{q/2}}+ \frac{\delta^2}{k^{q-1}} \right)$ for  $q \in [1, 2)$ for nonconvex composite problems, and of order  $\mathcal{O}\left(\frac{1}{k} + \frac{\delta}{k^{q/2}} \right)$ for  $q \in [0, 2)$ for convex composite problems in the form \eqref{eq:opt_prblm}. We also derive convergence rates {of order $\mathcal{O}(\frac{1}{k^2} + \frac{\delta} {k^{(3q-2)/2}})$} for a fast inexact proximal gradient algorithm  for solving the convex composite problem \eqref{eq:opt_prblm}.  Note that  our convergence rates are better as  $q$ increases. In particular, for the inexact proximal gradient algorithm the power of $\delta$ in the convergence estimate  is higher for $q \in (0,1)$ than for $q=0$, while for $q \geq 1$  the coefficients of $\delta$ diminishes with $k$. For the fast inexact proximal gradient method we show that there is no error accumulation for $q \geq 2/3$.  Hence, it is beneficial to consider an inexact  first-order oracle of degree $q > 0$, as this allows us to work with less accurate approximation of the (sub)gradient of $F$ when  $q$ is large.


\section{Notations and preliminaries}
\noindent In what follows $\mathbb{R}^n$ denotes the finite-dimensional Euclidean space endowed with the standard inner product $\langle s, x\rangle = s^Tx$ and the corresponding norm $\| s \| = \langle s,s\rangle^{1/2}$ for any $s\in\mathbb{R}^n$. For a proper lower semicontinuous convex function $h$ we denote its domain by $\text{dom}\;h = \lbrace x\in\mathbb{R}^n:\; h(x)< \infty \rbrace $ and its proximal operator as:
\vspace{-0.5cm}
\begin{align*}
\text{prox}_{\gamma h}(x) := \argmin_{y\in\text{dom}\;h} h(y) + \frac{1}{2\gamma}\| x - y\|^2.
\end{align*}

\vspace{-0.4cm}

\noindent \red{Next, we provide a few definitions and properties for subdifferential calculus in the nonconvex settings (see \cite{Mor:06,RoWe:98} for more details).
\begin{definition}
\textbf{(Subdifferential)}: Let $f: \mathbb{R}^n \to \bar{\mathbb{R}}$ be a proper lower semicontinuous function. For a given $x \in \text{dom} \; f$, the Fr$\acute{e}$chet subdifferential of $f$ at $x$, written $\widehat{\partial}f(x)$, is the set of all vectors $g_{x}\in\mathbb{R}^n$ satisfying:
\begin{equation*}
\lim_{x\neq y}\inf\limits_{y\to x}\frac{f(y) - f(x) - \langle g_{x}, y - x\rangle}{\norm{x-y}}\geq 0.
\end{equation*}
When $x \notin\text{dom} \; f$, we set $\widehat{\partial} f(x) = \emptyset$. The limiting-subdifferential, or simply the subdifferential, of $f$ at $x\in \text{dom} \, f$, written $\partial f(x)$, is defined as \cite{Mor:06}:
\begin{align*}
\partial f(x):= \left\{ g_{x}\in \mathbb{R}^n\!\!: \exists x^{k}\to x, f(x^{k})\to f(x) \; \text{and} \; \exists g_{x}^{k}\in\widehat{\partial} f(x^{k}) \;\; \text{such that} \;\;  g_{x}^{k} \to g_{x}\right\}. 
\end{align*} 
\end{definition}
\noindent Note that we have $\widehat{\partial}f(x)\subseteq\partial f(x)$  for each $x\in\text{dom}\,f$. For $f(x) = F(x) + h(x)$, if $F$ and $h$ are regular at $x\in\text{dom}f$, then we have:
\begin{align*}
    \partial f(x) = \partial F(x) + \partial h(x).
\end{align*}
(see Theorem 6 in \cite{Hir:79} for more
details). Further, if $f$ is proper, lower semicontinuous and convex, then \cite{RoWe:98}:
\begin{align*}
  \partial f(x) = \partial \hat{f}(x) = \{\lambda \in\mathbb{R}^n: f(y)\geq f(x) + \langle\lambda,y-x \rangle\;\; \forall y\in\mathbb{R}^n\}.  
\end{align*}
} 

\noindent A function $F : \mathbb{R}^n\to \mathbb{R}$ is $L_F$-smooth if it is differentiable and its gradient is $L_F$ Lipschitz, i.e., satisfying:
\begin{align*}
\| \nabla F(x) - \nabla F(y) \| \leq L_F \| x - y \|,\quad \forall x,y\in\mathbb{R}^n.
\end{align*}
It follows immediately that \cite{Nes:04}:
\begin{align}\label{eq:smth}
\vert F(x) - (F(y) + \langle \nabla F(y),x-y \rangle) \vert \leq \frac{L_F}{2} \| x - y \|^{2}, \quad \forall x,y\in\mathbb{R}^n. 
\end{align} 
Finally, let us recall the following classical weighted arithmetic-geometric mean inequality:
if $a,b$ are positive constants and $0\leq \alpha_1 ,\alpha_2 \leq 1$, such that $\alpha_1 + \alpha_2 = 1$, then  $a^{\alpha_1}b^{\alpha_2}\leq \alpha_1 a + \alpha_2 b$.  We will later use the following consequence
for  $\rho >0$, $a=\rho\norm{x-y}^2$,  $b=\frac{\delta_{q}^{\frac{2}{2-q}}}{\rho^{\frac{q}{2-q}}}$, $\alpha_1 = \frac{q}{2}$ and $\alpha_2 =  \frac{2 - q}{2}$:
\begin{align} \label{eq:apd1}
{\delta_q} \norm{x-y}^q \leq \frac{q\rho\norm{x-y}^2}{2} + \frac{(2-q)\delta_{q}^{\frac{2}{2-q}}}{2\rho^{\frac{q}{2-q}}}.
\end{align}

\section{Inexact first-order oracle of degree $q$}
\noindent In this section, we introduce our new inexact first-order oracle of degree $0\leq q <2$ and provide some nontrivial examples that fit into our framework. Our oracle can deal with general functions (possibly with unbounded domain),  unlike the previous results in \cite{DeGlNe:13,Dvur:17},  but requires  exact  zero-order information.  

\begin{definition}\label{eq:def1}
The function $F$ is equipped with an inexact  first-order $(\delta, L)$-oracle of degree $q \!\in\! [0,2)$ if for any $y \!\in\! \text{dom} f$ one can compute $ g_{\delta,L,q}(y) \in \mathbb{E}^{*}$ such that:
\begin{align}\label{eq:inx2}
 F(x) \!-\! \left(F(y) + \langle g_{\delta,L,q}(y),x-y \rangle\right) \!\leq\! \frac{L}{2}\| x-y \|^{2} + \delta\| x - y \|^q \;\;\;  \forall x \!\in\! \text{dom} f.
\end{align}
\end{definition}

\noindent To the best of our knowledge this definition of a first-order inexact oracle is new. \red{The motivation behind this definition is to introduce a versatile inexact first-order oracle framework that bridges the gap between exact oracle (exact gradient information, i.e., $q=2$) and the existing inexact first-order oracle definitions found in the literature (i.e., $q=0$). More specifically, when $q=2$, Definition \ref{eq:def1} aligns with established results for smooth functions under exact gradient information, while when $q=0$, our definition has been previously explored in the literature, see  \cite{DeGlNe:13,Dvur:17}.} Next, we provide several examples that satisfy Definition \ref{eq:def1} {naturally}, and then we provide theoretical results showing the advantages of this new inexact oracle over the existing ones from the literature.
\begin{example}(Smooth function with inexact first-order oracle).
\label{eq:ex1}
Let $F$ be differentiable and its gradient be Lipschitz continuous with constant  $L_F$ over $\text{dom}f$. Assume that for any $x\in\text{dom}\,f$, one can compute $g_{\Delta,L_F}(x)$, an approximation of the gradient $\nabla F(x)$ satisfying:
\begin{align}\label{eq:con_ex1}
\| \nabla F(x) - g_{\Delta,L_F}(x) \|\leq \Delta.
\end{align}
Then, $F$ is equipped with \red{$(\delta,L)$-oracle} of degree $q=1$ as in Definition \ref{eq:def1}, with  $\delta = \Delta$, $L=L_F$, and $g_{\delta,L,1}(x)= g_{\Delta,L_F}(x)$. Indeed, since $F$ is $L_F$-smooth, we get:
\begin{align*}
 F(y) - F(x) - \langle \nabla F(x),y-x \rangle&\leq \frac{L_F}{2}\| y - x \|^2.
\end{align*}
It follows that:
\begin{align*}
 F(y)\! -\! F(x)\! - \!\langle g_{\Delta,L_F}(x),y \!-\! x \rangle\! &\leq \frac{L_F}{2}\| y \!-\! x \|^2 \!+\! \| \nabla F(x) \!-\! g_{\Delta,L_F}(x) \|\,\| y \!-\! x \|\\
&\leq \frac{L_F}{2}\| y - x \|^2 + \Delta \| y - x \|. 
\end{align*}
which completes our statement.
Finite sum optimization problems appear widely in machine learning \cite{Bot:10} and deal with an objective $F(x): = \sum_{i=1}^{N} F_i(x)$, where $N$ is possibly large.  In the stochastic setting, we sample stochastic derivatives at each iteration in order to form a mini-batch approximation for the gradient of $F$. If we define:
\begin{align}\label{eq:sapl_stc}
g_{S}(x) = \frac{1}{\vert S\vert}\sum\limits_{j\in S} \nabla F_i(x),
\end{align}
where $S$ is a subset of $\{1,\ldots,N\}$, then condition \eqref{eq:con_ex1} holds \red{with} probability at least $1 - \Delta$ if the batch size $S$ satisfies $\vert S \vert = \mathcal{O}\left( \frac{\Delta^2}{L_{F}^2} + \frac{1}{N} \right)^{-1}$ (see Lemma 11 in \cite{AgKam:17}).   
\end{example}
\begin{remark}
This example has been also considered in \cite{DeGlNe:13,Dvur:17}. However, in these papers $\delta$ depends on the diameter of the domain of $f$, assumed to be bounded. Our inexact oracle is more general and doesn't require boundedness of the domain of $f$, {i.e.,} in our case $\delta = \Delta$, \red{while in \cite{DeGlNe:13,Dvur:17}, $\delta = 2\Delta D$, where $D$ is the diameter of the domain of $f$. Hence, our definition is more natural in this setting.}  
\end{remark}


\begin{example} 
(Computations at shifted points) \label{eq:ex2}
Let $F$ be differentiable with Lipschitz continuous gradient with constant $L_F$ over $\text{dom}f$. For any $x\in\text{dom}f$ we assume we can compute the exact value of the gradient, albeit evaluated at a shifted point $\Bar{x}$, different from $x$ and satisfying $\|x - \bar{x}\| \leq \Delta$. Then, $F$ is equipped with a $(\delta,L)$-oracle of degree $q=1$ as in Definition \ref{eq:def1}, with $g_{\delta,L,1}(x) = \nabla F(\bar{x})$, $L = L_F$ and $\delta = L_F \Delta$. Indeed, since $F$ is $L_F$ smooth, we have:
 \begin{align*}
     F(y)&\leq F(x) + \langle \nabla F(x),y-x \rangle + \frac{L_F}{2}\|y - x\|^2,\\
         & = F(x) + \langle \nabla F(\bar{x}),y-x \rangle + \langle \nabla F(x) -  \nabla F(\bar{x}),y-x \rangle  + \frac{L_F}{2}\|y - x\|^2,\\
         & \leq F(x) + \langle \nabla F(\bar{x}),y-x \rangle + \frac{L_F}{2}\|y - x\|^2 + L_F\|x - \bar{x}\|\|y - x\|,
 \end{align*}
where the second inequality follows from the Cauchy-Schwartz inequality. This proves our statement.
\end{example}
\begin{remark}
\red{This example was also considered in \cite{DeGlNe:13,Dvur:17}, with the corresponding   $(\delta,L)$-oracle having $\delta = L_F\Delta^2 $, $L = 2L_F$ and $q=0$. Note that our $L$ in  Definition \ref{eq:def1} is twice smaller than the corresponding $L$ in \cite{DeGlNe:13,Dvur:17}.}    
\end{remark}


\begin{example}\red{(Accuracy measures for approximate solutions)
\label{eq:ex3}
Let us consider a $F$ that is $L_F$ smooth, given by:
\begin{align*}
    F(x) = \max_{u\in U} \psi(x,u):=  \max_{u\in U} G(u) + \langle Au,x\rangle,
\end{align*}
where $A:\mathbb{E}\to \mathbb{E}^*$ is a linear operator and $G(\cdot)$ is a differentiable strongly concave function with concavity parameter $\kappa >0$. Under these assumptions, the maximization problem $\max_{u\in U} \psi(x,u)$ has only one optimal solution $u^{*}(x)$ for a given $x$. Moreover, $F$ is convex and smooth with Lipschitz continuous gradient $\nabla F(x) = \nabla_x \psi(x,u^{*}(x)) = Au^{*}(x)$ having Lipschitz constant $L_F  = \frac{1}{\kappa}\|A\|^2$ \cite{DeGlNe:13}. Suppose that for any $x\in\text{dom}f$, one can compute $u_x$ an approximate minimizer of $\psi(x,u)$ such that $\|u^{*}(x) - u_x\|\leq \Delta$. Then, $F$ is equipped with $(\delta,L)$-oracle of degree $q=1$ with $\delta = \Delta \|A\|$, $L = L_F$ and $g_{\delta, L, 1}(x) = Au_x$. Indeed,  since $F$ has Lipschitz-continuous gradient, we have:
\begin{align*}
    F(y) &\leq F(x) + \langle \nabla F(x),y - x \rangle + \frac{L_F}{2}\| y - x \|^2,\\
         & = F(x) + \langle \nabla_x \psi(x,u^{*}(x)), y - x \rangle + \frac{L_F}{2}\| y - x\|^2,\\
         & = F(x) + \langle Au^{*}(x), y - x \rangle + \frac{L_F}{2}\|y - x\|^2,\\
         & = F(x) + \langle Au_x
         ,y - x \rangle + \langle A(u^{*}(x) - u_x), y - x \rangle + \frac{L_F}{2}\|y - x\|^2,\\
         & \leq  F(x) + \langle Au_x
         ,y - x \rangle +  \|A\|\|u^{*}(x) - u_x\|\| y - x\| + \frac{L_F}{2}\|y - x\|^2,\\
          & \leq  F(x) + \langle Au_x
         ,y - x \rangle +  \Delta\|A\|\| y - x\| + \frac{L_F}{2}\| y - x\|^2.
\end{align*}
Hence, our statement follows.}
\end{example}

\begin{remark}
\red{This example was also considered in \cite{DeGlNe:13} with the corresponding $(\delta,L)-$ oracle having $\delta = \Delta $, $L = 2L_F$ and $q=0$, while in our case, we have $\delta = \Delta \|A\|$, $L = L_F$ and $q = 1$.}    
\end{remark}


\begin{example}(Weak level of smoothness)
\label{eq:ex4}
Let $F$ be a proper lower semicontinuous function with the subdifferential $\partial F(x) $ nonempty for all $x\in\text{dom}f$. Assume that $F$ satisfies the following Hölder condition with $H_{\nu}<\infty$:
\begin{align}\label{eq:hold}
\norm{g(x) - g(y)}\leq H_{\nu}\norm{y-x}^\nu,
\end{align}
for all $g(x)\in \partial F(x)$, $g(y)\in\partial F(y)$, where $x,y\in\text{dom}\,f$ and $\nu\in[0,1]$. Then, $F$ is equipped with \red{$(\delta,L)$-oracle} of degree $q$ as in Definition \ref{eq:def1}, with $g_{\delta,L,q}(x) \in\partial F(x)$, for any arbitrary degree $0 \leq q<1+\nu$ and any accuracy  $\delta >0$, and a constant $L$ depending on $\delta$ given by:
\[L(\delta) = \frac{1+\nu-q}{2-q}\left( \frac{H_{\nu}}{1+\nu}\right)^{\frac{2-q}{1+\nu-q}}\left(\frac{1-\nu}{\delta(2-q)}\right)^{\frac{1-\nu}{1+\nu - q}}.\]

\noindent Indeed, we have from Hölder condition~\cite{Nes:04}:
\begin{align*}
F(x) - F(y) - \langle g(y),x-y \rangle\leq \frac{H_{\nu}}{1 + \nu} \norm{x-y}^{1+\nu}.
\end{align*}
For any given $\delta >0$, we compute $L(\delta)$ such that the following inequality holds:
\begin{align*}
\frac{H_{\nu}}{1 + \nu}\norm{x-y}^{1+\nu}\leq\frac{L(\delta)}{2} \norm{x-y}^2 + \delta \norm{x-y}^q.
\end{align*}
Denote $r=\norm{x-y}$ and let $\lambda\in (0,1)$. Using the weighted arithmetic-geometric mean
inequality with $\alpha_1 = \lambda$ and $\alpha_2 = 1-\lambda$, we have:
\begin{align*}
\frac{L(\delta)r^2}{2} + \delta r^q &=\lambda\frac{L(\delta)}{2\lambda}r^2 + (1-\lambda)\frac{\delta}{1-\lambda} r^q\\ 
      						  &\geq \left(\frac{L(\delta)}{2\lambda} r^2\right)^{\lambda}\left(\frac{\delta}{1 \!-\! \lambda}r^q \right)^{1 \!-\! \lambda} \!=\! \left(\frac{L(\delta)}{2\lambda}\right)^{\lambda}\left(\frac{\delta}{1 \!-\! \lambda}\right)^{1 \!-\! \lambda}\!\! r^{2\lambda + q(1 \!-\! \lambda)}.\\
\end{align*}
Thus $\frac{H_{\nu}}{1+\nu}=\left(\frac{L(\delta)}{2\lambda}\right)^{\lambda}\left(\frac{\delta}{1-\lambda}\right)^{1-\lambda}$ and $1 + \nu=2\lambda + q(1-\lambda)$. It follows that $\lambda = \frac{1+\nu-q}{2-q}$, $1-\lambda = \frac{1-\nu}{2-q}$ and  $\frac{1}{\lambda} - 1 = \frac{1-\nu}{1+\nu - q}$.
Hence, for a given positive $\delta$  one may choose:
\begin{align*}
L(\delta) \!=\! 2\lambda \left(\frac{H_{\nu}}{1 \!+\! \nu}\right) ^{\frac{1}{\lambda}}\left(\frac{1 \!-\! \lambda}{\delta}\right)^{\frac{1}{\lambda} \!-\! 1} \!=\!  \frac{1 \!+\!  \nu \!-\! q}{2 \!-\! q}\left(\frac{H_{\nu}}{1 \!+\! \nu}\right)^{\frac{2 \!-\! q}{1 \!+\! \nu \!-\! q}}\left(\frac{1 \!-\! \nu}{\delta(2 \!-\! q)}\right)^{\frac{1 \!-\! \nu}{1 \!+\! \nu \!-\! q}},
\end{align*}
and this is our statement. \red{ Note that if $\nu > 0$, then we have $\partial F(x) = \lbrace \nabla F(x) \rbrace$ for all $x$ and thus $F$ is differentiable. Indeed, letting $y=x$ in \eqref{eq:hold} we get:
    $g(x) = \Bar{g}(x)$.
This implies that the set $\partial F(x)$ has a single element, thus $F$ is differentiable}. This example covers large classes of functions. Indeed, when $\nu = 1$, we get functions with Lipschitz-continuous subgradient. For $\nu < 1$, we get a weaker level of smoothness. In particular, when $\nu = 0$, we obtain functions whose subgradients have bounded variation. Clearly, the latter class includes functions whose subgradients are uniformly bounded by $M$ (just take $H_0 = 2M$). It also covers functions smoothed by local averaging and Moreau–Yosida regularization (see \cite{DeGlNe:13} for more details).  We believe that the readers may find other examples that satisfy our Definition \ref{eq:def1} of an inexact first-order  oracle of degree $q$.
\end{example}


\section{Inexact proximal gradient method}
\noindent In this section, we introduce an inexact proximal gradient  method based on the previous inexact oracle definition for solving (non)convex composite minimization problems \eqref{eq:opt_prblm}.  We derive complexity estimates for this algorithm  and study the dependence between the accuracy of the oracle and the desired accuracy of the gradient or of the objective function. Hence, we consider the following Inexact Proximal Gradient Method (I-PGM).
\begin{algorithm}\label{algo1}
\caption{Inexact proximal gradient method (I-PGM)}
\begin{enumerate}
  \item[1.] Given $x_0 \in \text{dom}\; h$ and $0\leq q < 2$.
   \item[] For $k\geq 0$ do:
  \item[2.] Choose $\delta_k$, $L_k$ and 
        $\alpha_k$. Obtain $g_{\delta_k,L_k,q}(x_k)$.
  \item[3.] Compute $x_{k+1} = \text{prox}_{\alpha_k
h}\left(x_{k}-\alpha_k g_{\delta_k,L_k,q}(x_k)\right)$.
\end{enumerate}
\end{algorithm}
\noindent Note that Algorithm 1 is an inexact proximal gradient method, where the inexactness comes from the approximate computation of the (sub)gradient of $F$, denoted $g_{\delta_k,L_k,q}(x_k)$. In the next sections we analyze the convergence behavior of this algorithm when $g_{\delta_k,L_k,q}(x_k)$ satisfies Definition \ref{eq:def1}.


\subsection{Nonconvex convergence analysis}
\noindent In this section we consider a nonconvex function $F$ that admits an inexact first-order $(\delta,L)$-oracle of degree $q$ as in Definition \ref{eq:def1}. Using this definition and inequality \eqref{eq:apd1}, for all $\rho>0$ we get the following upper bound:
\begin{align}\label{eq:inx3}
F(x) - \Big(F(y) + \langle g_{\delta,L,q}(y),x-y \rangle\Big)\leq \frac{L + q\rho}{2}\norm{x-y}^{2} + \frac{(2-q)\delta^{\frac{2}{2-q}}}{2\rho^{\frac{q}{2-q}}}.
\end{align}
This inequality will play a key role in our convergence analysis. We define the \textit{gradient mapping} at iteration $k$ as $g_{\delta_k,L_k,q}(x_k) + p_{k+1}$, where $p_{k+1}\in\partial h(x_{k+1})$ such that $ g_k + p_{k+1} = -\frac{1}{\alpha_k}(x_{k+1} - x_k)$ (i.e., $p_{k+1}$ is the subgradient of $h$ at $x_{k+1}$ coming from the optimality condition of the prox at $x_k$). Next we analyze the global convergence of I-PGM in the norm of the gradient mapping. We have the following theorem:


\begin{theorem}\label{th:2}
Let $F$ be a nonconvex function admitting a $(\delta_k,L_k)$-oracle of degree $q\in [0,2)$ at each iteration $k$, \red{with $\delta_k \geq 0$ and $L_k > 0$ for all $k\geq 0$}. Let $(x_k)_{k\geq 0}$ be generated by I-PGM and assume that $\alpha_k \leq \frac{1}{L_k + q\rho}$, for some arbitrary parameter $\rho >0$. Then,  there exists $p_{k+1} \in\partial h(x_{k+1})$ such that:
\begin{align}\label{eq:th}
\sum_{j=0}^{k}\alpha_j \norm{g_{\delta_{j},L_{j},q}(x_j) + p_{j+1} }^2\leq f(x_0) - f_\infty + \frac{\sum_{j=0}^{k}(2-q)\delta_j^{\frac{2}{2-q}}}{2\rho^{\frac{q}{2-q}}}.
\end{align}  
\end{theorem}
\begin{proof}
Denote $g_{\delta_k,L_k,q}(x_k)= g_k$. From the optimality conditions of the proximal operator defining $x_{k+1}$, we have:
\begin{align*}
g_k + p_{k+1} = -\frac{1}{\alpha_k} (x_{k+1} - x_k) .
\end{align*} 
Further, from inequality \eqref{eq:inx3}, we get:
\begin{align*}
&F(x_{k+1}) \leq F(x_k) + \langle g_k,x_{k+1} - x_k \rangle + \frac{L_k + q\rho}{2}\norm{x_{k+1} - x_k}^2 + \frac{(2-q)\delta_k^{\frac{2}{2-q}}}{2\rho^{\frac{q}{2-q}}}\\
           & = F(x_k) \!+\! \langle g_k \!+\! p_{k+1},x_{k+1} \!-\! x_k \rangle  \!-\! \langle p_{k+1},x_{k+1} \!-\! x_k \rangle \!+\!  \frac{L_k + q\rho}{2}\norm{x_{k+1} \!-\! x_k}^2\\
           		   &\quad + \frac{(2-q)\delta_k^{\frac{2}{2\!-\!q}}}{2\rho^{\frac{q}{2-q}}}\\
		   & \leq F(x_k)  \!-\!\alpha_k \left( 1 - \frac{(L_k \!+\! q\rho)\alpha_k }{2}\right) \| g_k \!+\! p_{k+1} \|^2 \!+\! h(x_k) \!-\! h(x_{k+1}) \!+\! \frac{(2\!-\!q)\delta_k^{\frac{2}{2-q}}}{2\rho^{\frac{q}{2-q}}}\\
		   &\leq F(x_k) -\frac{\alpha_k}{2} \| g_k + p_{k+1}\|^2 + h(x_k) - h(x_{k+1})  + \frac{(2\!-\!q)\delta_k^{\frac{2}{2-q}}}{2\rho^{\frac{q}{2-q}}},  
\end{align*}
where the second inequality follows from the convexity of $h$ and $p_{k+1}\in\partial h(x_{k+1})$, and the last inequality follows from the definition of $\alpha_k$. Hence, we get that:
\begin{align*}
f(x_{k+1})\leq f(x_k) - \frac{\alpha_k}{2} \norm{g_k \!+\! p_{k+1}}^2 + \frac{(2-q)\delta_k^{\frac{2}{2-q}}}{2\rho^{\frac{q}{2-q}}}.
\end{align*}
Summing up this inequality from $j=0$ to $j=k$ \red{and using the fact that $f(x_{k+1}) \geq f_\infty$, where recall that  $f_\infty$ denotes a finite lower bound for the objective function}, we get:
\begin{align*}
\sum_{j=0}^{k} \frac{\alpha_j}{2} \norm{g_j + p_{j+1}}^2 &\leq f(x_0) - f(x_{k+1}) + \frac{\sum_{j=0}^{k}(2-q)\delta_j^{\frac{2}{2-q}}}{2\rho^{\frac{q}{2-q}}}\\
            & \leq f(x_0) - f_\infty + \frac{\sum_{j=0}^{k}(2-q)\delta_j^{\frac{2}{2-q}}}{2\rho^{\frac{q}{2-q}}}.
\end{align*}
Hence, our statement follows.
\end{proof}

\noindent For a particular choice of the algorithm parameters, we can get simpler convergence estimates. 

\begin{theorem}\label{th:1}
Let the assumptions of Theorem \ref{th:2} hold and consider for all $k \geq 0$: 
$$ L_k = L,\; \delta_k = \frac{\delta}{(k+1)^{\frac{\beta(2-q)}{2}}},\;\alpha_k = \frac{1}{(L + q\rho)(k+1)^{\zeta}}, \;\emph{where}\; \beta, \zeta \in[0,1).$$ 
 Then, we have:
\begin{align}\label{eq:th1}
&\min_{j=0:k} \norm{g_j + p_{j+1}}^2 \leq \frac{2(L \!+\! q\rho)(f(x_0) \!-\! f_\infty)}{(1 \!-\! \zeta)(k+1)^{1-\zeta}} \!+\! \frac{(2\!-\!q)(L \!+\! q\rho)\delta^{\frac{2}{2-q}} }{(1-\zeta)(1-\beta)\rho^{\frac{q}{2-q}} (k+1)^{\beta-\zeta}}.
\end{align}
\end{theorem}

\begin{proof}

Taking the minimum in the inequality \eqref{eq:th}, we get:
\begin{align*}
\min_{j=0:k} \norm{g_j + p_{j+1}}^2 &\leq \frac{2(f(x_0) - f_\infty)}{\sum_{j=0}^{k} \alpha_j} + \frac{\sum_{j=0}^{k}(2-q)\delta_j^{\frac{2}{2-q}}}{\rho^{\frac{q}{2-q}}\sum_{j=0}^{k-1}\alpha_j}.
\end{align*}
Further, since we have:
\begin{align*}
 \sum_{j=0}^{k} \frac{1}{(L+q\rho)(j+1)^\zeta} = \sum_{j=1}^{k+1} \frac{1}{(L+q\rho)j^\zeta},    
\end{align*}
and  similarly for $\delta_j$, we get:
\begin{align*}
\min_{j=0:k} \norm{g_j + p_{j+1}}^2 & \leq \frac{2(L + q\rho)(f(x_0) - f_\infty)}{\sum_{j=1}^{k+1} \frac{1}{j^\zeta}} + \frac{(2 - q)(L + q\rho) \delta^{\frac{2}{2-q}} \sum_{j=1}^{k+1} \frac{1}{j^\beta}}{\rho^{\frac{q}{2-q}} \sum_{j=1}^{k+1} \frac{1}{j^\zeta}}. 
\end{align*}
Since $0 \leq \zeta <1$, then we have for all $k\geq 0$:
\begin{align*}
(1-\zeta)(k+1)^{1-\zeta}\leq\frac{(k+2)^{1-\zeta} - 1}{1-\zeta} &= \int_{1}^{k+2}\frac{1}{u^{\zeta}} du\\
&\leq \sum_{j=1}^{k+1} \frac{1}{j^\zeta}
\leq  \int_{1}^{k+1} \left(\frac{1}{u^\zeta} \right)du + 1 \leq  \frac{(k+1)^{1-\zeta}}{1-\zeta}.   
\end{align*}
It follows that for all $k\geq 0$:
\begin{align*}
\min_{j=0:k} \norm{g_j + p_{j+1}}^2 &\leq \frac{2(L \!+\! q\rho)(f(x_0) \!-\! f_\infty)}{(1 \!-\! \zeta)(k+1)^{1-\zeta}} +\frac{(2\!-\!q)(L \!+\! q\rho)\delta^{\frac{2}{2-q}} }{(1-\zeta)(1-\beta)\rho^{\frac{q}{2-q}} (k+1)^{\beta-\zeta}}.
\end{align*}
Hence, our statement follows.
\end{proof}

\noindent Let us analyze in more details the bound from Theorem \ref{th:1}. For  simplicity, consider the case $q = 1$ (see Example \ref{eq:ex1}). Then,  we have:

 \begin{align*}
\min_{j=0:k} \norm{g_j + p_{j+1}}^2 &\leq \frac{2(L + \rho)(f(x_0) - f_\infty)}{(1-\zeta)(k+1)^{1 - \zeta}} + \frac{(L + \rho)\delta^2}{\rho(1-\zeta)(1-\beta) (k+1)^{\beta-\zeta}}\\
& = \frac{2L(f(x_0) - f_\infty)}{(1-\zeta)(k+1)^{1 - \zeta}} + \frac{2\rho(f(x_0) - f_\infty)}{(1-\zeta)(k+1)^{1 - \zeta}} \\
&\quad + \frac{L \delta^2}{\rho(1-\zeta)(1-\beta)(k+1)^{\beta-\zeta}} + \frac{\delta^2}{ (1-\zeta)(1-\beta)(k+1)^{\beta-\zeta}}. 
\end{align*}
Denote $\Delta _0:= f(x_0) - f_\infty$. Since parameter $\rho > 0$ is a degree of freedom,  minimizing the right hand side of the previous relation w.r.t. $\rho$ we get an optimal choice  $\rho = \frac{\delta\sqrt{L}}{\sqrt{2\Delta_0(1-\beta)}}(k+1)^{\frac{1 - \beta}{2}}$. Hence, replacing this expression for $\rho$ in the last inequality, we get:
\begin{align*}
 \min_{j=0:k} \norm{g_j + p_{j+1}}^2 &\leq \frac{2L\Delta_0}{(1-\zeta)(k+1)^{1-\zeta}}  + \frac{2\delta\sqrt{2L\Delta_0}}{((1-\zeta)\sqrt{1-\beta})(k+1)^{\frac{1+\beta}{2} - \zeta}}\\
 &+ \frac{\delta^2}{(1-\zeta)(1-\beta)(k+1)^{\beta-\zeta}}.
\end{align*}
This bound is of order $\mathcal{O}\left(\frac{1}{k^{1-\zeta}} + \frac{\delta}{k^{\frac{1+\beta}{2} - \zeta}} + \frac{\delta^2}{k^{\beta-\zeta}}\right)$. Note that, if $\beta > \zeta$, the gradient mapping  $\min_{j=0:k} \norm{g_j + p_{j+1}}^2$ converges  regardless of the accuracy of the oracle  $\delta$ and the convergence rate is of order $\mathcal{O}(k^{-\text{min}(1-\zeta,\beta - \zeta)})$ (since we always have $\frac{1+\beta}{2} - \zeta \geq \beta - \zeta$).  Note that this is not the case for $q=0$, where the convergence rate is of order $\mathcal{O}\left(\frac{1}{k} + \delta\right) $, see also \cite{Dvur:17}.  The following corollary provides a convergence rate for general $q$, but for a particular choice of the parameters $\zeta$ and $\beta$.


\begin{corollary}\label{eq:col1}
Let the assumptions of Theorem \ref{th:1} hold and let assume that $\zeta = \beta = 0$. Then, we have the following convergence rates: 
\begin{enumerate}
\item If $0\leq q < 2$ and $\rho = L$, then $\delta_k = \delta$, $\alpha_k = \frac{1}{L+qL}$ and
\begin{align*}
&\min_{j=0:k} \norm{g_j + p_{j+1}}^2 \leq \frac{2(q+1)L\Delta_0}{k+1} +(q+1)(2-q)L^{\frac{2-2q}{2-q}}\delta^{\frac{2}{2-q}} \quad \forall k \geq 0. 
\end{align*}

\item If $1\leq q < 2$, fixing the number of iterations $k$ and taking  $\rho  = \frac{L^{\frac{2-q}{2}}\delta}{(2\Delta_0)^{\frac{2-q}{2}}}(k+1)^{\frac{2-q}{2}}$, then $\delta_j = \delta$, $\alpha_j =\frac{1}{L + q\rho}$ for all $j=0:k$ and
\begin{align*}
&\min_{j=0:k} \norm{g_j + p_{j+1}}^2\\
&\leq \frac{2L\Delta_0}{k+1}\! +\! \frac{L^{\frac{2-q}{2}}(2\Delta_0)^{\frac{q}{2}}\delta\! +\! (2-q)\delta L^{1-\frac{q}{2}}(2\Delta_0)^{\frac{q}{2}}}{(k+1)^{\frac{q}{2}}}\!+\! \frac{q(2-q)\delta^2 L^{1-q}(2\Delta_0)^{q-1}}{(k+1)^{q-1}}.
\end{align*}
\end{enumerate}
\end{corollary}

\begin{proof}
Replacing $\zeta = \beta = 0$ in inequality \eqref{eq:th1}, we get: 
\begin{align*}
\min_{j=0:k} \norm{g_j + p_{j+1}}^2 &\leq \frac{2(L + q\rho)\Delta_0}{k+1} + \frac{(2-q)(L + q\rho)\delta^{\frac{2}{2-q}}}{\rho^{\frac{q}{2-q}}}\\\nonumber
						    & = \frac{2L\Delta_0}{k+1} + \frac{2q\rho\Delta_0}{k+1} + \frac{(2-q)L\delta^{\frac{2}{2-q}}}{\rho^{\frac{q}{2-q}}} + \frac{q(2-q) \delta^{\frac{2}{2-q}}}{\rho^{\frac{2q-2}{2-q}}}.
\end{align*}
\red{If $0\leq q <2$}, then taking $\rho = L $ in the last inequality we get the first statement. Further, if $1\leq q <2$, minimizing over $\rho$ the second and the third terms of the right side of the last inequality yields the optimal choice $\rho = \frac{L^{\frac{2-q}{2}}\delta}{(2\Delta _0)^{\frac{2-q}{2}}}(k+1)^{\frac{2-q}{2}}$. Replacing this expression for $\rho$ in the last inequality, we get:
\begin{align*}
&\min_{j=0:k} \norm{g_j + p_{j+1}}^2\\
&\leq  \frac{2L\Delta_0}{k+1}\! +\! \frac{L^{\frac{2-q}{2}}(2\Delta_0)^{\frac{q}{2}}\delta\! +\! (2-q)\delta L^{1-\frac{q}{2}}(2\Delta_0)^{\frac{q}{2}}}{(k+1)^{\frac{q}{2}}}\!+\! \frac{q(2-q)\delta^2 L^{1-q}(2\Delta_0)^{q-1}}{(k+1)^{q-1}},
\end{align*}
and this is the second statement.
\end{proof}

\begin{remark}
\red{Let us analyse in more details this convergence rate for Example \ref{eq:ex1}. For $q = 0$, we have that $\delta = 2D\Delta$ and $L = L_F$, where $D$ is the diameter of $\text{dom}f$. Hence, the convergence rate in this case becomes:
\begin{align*}
 \min_{j=0:k} \norm{g_j + p_{j+1}}^2 \leq  \frac{4L_F \Delta_0}{k+1} + 4DL_F\Delta.    
\end{align*}
On the other hand, for $q=1$, we have $\delta = \Delta$ and $L = L_F$. Thus, we get the following convergence rate:
\begin{align*}
 &\min_{j=0:k} \norm{g_j + p_{j+1}}^2 \leq \frac{4 L_F\Delta_0}{k+1} +2\Delta^2.    
\end{align*}
Hence, if we want to achieve $\min_{j=0:k} \norm{g_j + p_{j+1}}^2 \leq \epsilon$, for $q=0$ we impose $ 4DL_F\Delta \leq \epsilon/2$, which implies that one needs to compute an approximate gradient with accuracy $\Delta = \mathcal{O}(\epsilon)$, while for $q=1$ we impose $ 2\Delta^2 \leq \epsilon/2$, meaning that one only needs to compute an approximate gradient with accuracy $ \Delta =  \mathcal{O}(\epsilon^{1/2})$. Hence, for this example, it is more natural to use our inexact first-order oracle definition for $q=1$ than for $q=0$, since it requires less accuracy for approximating the true gradient.}    
\end{remark}

\noindent  { Note that in the second result of Corollary \ref{eq:col1}, the parameter $\rho$ depends on the difference $\Delta_0 = f(x_0) - f_\infty$, and, usually, $f_\infty$ is unknown. In practice, we can approximate $\Delta_0$ by using an estimate for $f_\infty$ in place of its exact value.  For example, one can consider $\Delta_0^k = f(x_0) - f_\text{best}^k$, where $f_\text{best}^k = \min_{j=0:k} f(x_j) - \varepsilon_k$ for some $\varepsilon_k \geq 0$, see \cite{GeBe:99}. Under this setting, the sequence $\varepsilon_k$ and the iterates of I-PGM corresponding to the case of the second result of Corollary \ref{eq:col1} are updated as follows:
\begin{algorithm}
\caption{Adaptive I-PGM algorithm when $f_\infty$ is unknown}
\begin{enumerate}
\item Given $\varepsilon_0 >0$ and $f_\text{best}^0 = f(x_0) - \varepsilon_0$. \\
For $k\geq 0$ do:
\item Compute $x_{k+1}$ by I-PGM with $\Delta_0^k = f(x_0) - f_\text{best}^k$.\\
\text{While $ f(x_{k+1}) < f_\text{best}^k$}  
\begin{enumerate}
\item Set $\varepsilon_k = 2\times \varepsilon_k$ and update $f_\text{best}^k = \min\limits_{j=0:k} f(x_j) - \varepsilon_k$.
\item Re-compute $x_{k+1}$ by I-PGM method. 
\end{enumerate}
\text{End While}
\item Set $\varepsilon_{k+1} = \frac{\varepsilon_k}{2}$.
\end{enumerate}
\end{algorithm}
}

\noindent { This process is well defined, i.e., the "while" step finishes in a finite number of iterations. Indeed, one can observe that if $\varepsilon_k \geq \min_{j=0:k} f(x_j) - f_\infty$ then $\varepsilon_k \geq \min_{j=0:k} f(x_j) - f(x_{k+1})$, which implies that $f(x_{k+1}) \geq f_\text{best}^k$. Additionally, we have $\varepsilon_{k} \leq 2(\min_{j=0:k} f(x_j) - f_\infty)$  for all $k\geq 0$. Hence, we can still derive a convergence rate for the second result of Corollary \ref{eq:col1} using this adaptive process since one can observe that:
\begin{align*}
f(x_0) - f(x_{k+1}) &\leq f(x_0) - f_\text{best}^k  = \Delta_0^k.
\end{align*}
Additionally, we have the following bound on $\Delta_0^k$:
\begin{align*}
\Delta_0^k &\leq  f(x_0) - \min\limits_{j=0:k}f(x_j) + 2(\min_{j=0:k} f(x_j) - f_\infty)\\
& = (f(x_0) - f_\infty) +  (\min\limits_{j=0:k}f(x_j) - f_\infty). 
\end{align*}
Hence, we can replace in \eqref{eq:th} the difference $\Delta_0 = f(x_0) - f_\infty$ with $\Delta_0^k$ and then the second statement of Corollary \ref{eq:col1} remains valid with $\Delta_0^k$ instead of $\Delta_0$.} 


\begin{remark}
We observe that for $q=0$ we recover the same convergence rate as in \cite{Dvur:17}. However, our result does not require the boundedness of the domain of $f$, while in \cite{Dvur:17} the rate depends explicitly of the diameter of the domain of $f$. Moreover, for $q >0$ our convergence bounds are better than in \cite{Dvur:17}, i.e., the coefficients of the terms in $\delta$ are either smaller or even tend to zero, while in \cite{Dvur:17} they are always constant.
\end{remark}

\noindent Further, let us consider the case of Example \ref{eq:ex4}, where $F$ satisfies the Hölder condition with constant $\nu\in (0,1]$ and $\beta=\zeta=0$. We have shown that for any $\delta >0 $ this class of functions can be equipped with a \red{$(\delta,L)$-oracle} of degree $q< 1+\nu$ with $L= C(H_{\nu},q)\left(\frac{1}{\delta}\right)^{\frac{1-\nu}{1+\nu - q}}$ (see Example \ref{eq:ex4} for the expression of the constant $C(H_\nu,q)$). In view of the first result of Corollary \ref{eq:col1}, after $k$ iterations, we have:
\begin{align*}
\min_{j=0:k} \norm{g_j + p_{j+1}}^2 &\leq 
\frac{2(q+1)\Delta_0 L}{k+1} +(q+1)(2-q)L^{\frac{2-2q}{2-q}}\delta^{\frac{2}{2-q}}\\
& = \frac{C_1}{k+1} \left(\frac{1}{\delta}\right)^{\frac{1-\nu}{1+\nu-q}} + C_2\left(\frac{1}{\delta}\right)^{\frac{(1-\nu)(2-2q)}{(1+\nu-q)(2-q)}}\delta^{\frac{2}{2-q}}\\
& = \frac{C_1}{k+1}\delta^{-\frac{1-\nu}{1+\nu-q}} + C_2 \delta^{-\frac{(1-\nu)(2-2q)}{(1+\nu-q)(2-q)}  + \frac{2}{2-q}}\\
& = \frac{C_1}{k+1}\delta^{-\frac{1-\nu}{1+\nu-q}} + C_2 \delta^{\frac{2\nu}{1+\nu-q}},
\end{align*}
where $C_1: = 2(q+1)\Delta_0 C(H_\nu,q)$ and $C_2 = (q+1)(2-q)C(H_\nu,q)^{\frac{2-2q}{2-q}}$. Since in this example we can choose  $\delta$,  its optimal value  can be computed from the following equation:
\begin{align*}
-\frac{C_1(1 - \nu )}{(1 + \nu - q)}\frac{1}{(k+1)}\delta^{\frac{q-2}{1 + \nu - q}} + \frac{2\nu C_2}{1+\nu -q}\delta^{\frac{-1 + \nu + q}{1 + \nu - q}} = 0.
\end{align*}
Hence, we get:
\begin{align*}
\delta = C_3 (k+1)^{-\frac{1+\nu -q}{1+ \nu }},
\end{align*}
where $C_3 = \left(\frac{2\nu C_2}{(1-\nu)C_1}\right)^{-\frac{1+\nu -q}{1+ \nu }}$. Thus, replacing this optimal  choice of $\delta$ in the last inequality, we get:
\begin{align*}
\min_{j=0:k} \norm{g_j + p_{j+1}}^2 &\leq C_1C_3\left((k+1)^{-\left(1 - \frac{1-\nu}{1 +\nu}\right)}\right) + C_2C_3\left((k+1)^{-\frac{2\nu}{ 1 + \nu } }\right)\\
     & = \frac{C_3(C_1 + C_2)}{(k+1)^{\frac{2\nu}{1 + \nu}}}.
\end{align*}

\begin{remark}
Note that our convergence rate of order $\mathcal{O}(k^{-\frac{2\nu}{1+\nu}})$ for Algorithm 1 (I-PGM) for nonconvex problems having the first term $F$ with a Hölder continuous gradient (Example \ref{eq:ex4}) recovers the rate obtained in \cite{Dvur:17} under the same settings. 
\end{remark}

\noindent Finally, let us now show that when the gradient mapping is small enough, i.e., $\| g_k + p_{k+1} \|$ is small, $x_{k+1}$ is a good approximation for a stationary point of problem \eqref{eq:opt_prblm}. Note that any choice 
$\alpha_k \leq \frac{1}{L + q\rho_k} $ yields:
\begin{align*}
\| x_{k+1} - x_{k} \|\leq \frac{1}{L} \left\Vert \frac{1}{\alpha_k}(x_{k+1} - x_k) \right\Vert = \frac{1}{L}\|g_k + p_{k+1}\|.
\end{align*}
Hence, if the  gradient mapping is small, then the norm of the difference $\| x_{k+1} - x_k \|$ is also small.

\begin{theorem}
Let $(x_{k})_{k\geq 0}$ be generated by I-PGM and let $p_{k+1} \in  \partial h(x_{k+1})$. Assume that we are in the case of Example \ref{eq:ex1}. Then, we have:
\begin{align*}
\emph{dist}(0,\partial f(x_{k+1}))\leq \| g_{\Delta,L_F,q}(x_k) + p_{k+1}\| + L_F \| x_{k+1} - x_k\|  + \Delta .
\end{align*}
Further, if we are in the case of Example \ref{eq:ex4}, then we have:
\begin{align*}
 \emph{dist}(0,\partial f(x_{k+1}))\leq \| g(x_k) + p_{k+1} \| + H_\nu \| x_{k+1} - x_k \|^{\nu}, \;\;  g(x_k)\in\partial F(x_k).
\end{align*}
\end{theorem}

\begin{proof}
Let us consider Example \ref{eq:ex1}, where $F$ is $L_F$ smooth and $h$ is convex. Since $\nabla F(x_{k+1}) + p_{k+1} \in\partial f(x_{k+1})$, then we have: 
\begin{align*}
&\| \nabla F(x_{k+1}) + p_{k+1} \|\\ 
& \leq \| g_{\Delta,L_F,q}(x_k) + p_{k+1} \| + \| \nabla F(x_k) - g_{\Delta,L_F,q}(x_k) \| + \| \nabla F(x_{k+1}) - \nabla F(x_k) \|\\
           &\leq  \| g_{\Delta,L_F,q}(x_k) + p_{k+1} \| + \Delta
 + L_F \| x_{k+1} - x_k \|.
\end{align*}
Further, let us assume that we are in the case of Example \ref{eq:ex4}. Then, we have $g(x_k) \in \partial F(x_k)$. Further, let $ g(x_{k+1}) \in \partial F(x_{k+1})$, then we get:
\begin{align*}
\| g(x_{k+1}) + p_{k+1} \|& \leq \| g(x_k) + p_{k+1} \| + \| g(x_{k+1}) - g(x_k) \|\\
           &\leq  \| g(x_k) + p_{k+1} \| + H_\nu \| x_{k+1} - x_k \|^{\nu}.
\end{align*} 
This proves our statements.
\end{proof}

\noindent Thus, for $\|\frac{1}{\alpha_k}(x_{k+1} - x_k) \| = \|g_k + p_{k+1}\|$  small, $x_{k+1}$ is an approximate stationary point of problem \eqref{eq:opt_prblm}.  Note that  our convergence rates from this section are better as   $q$ increases, i.e., the terms depending on $\delta$ are smaller for $q>0$ than for $q=0$. In particular, the power of $\delta$ in the convergence estimate is higher for $q \in (0,1)$ than for $q=0$, while for $q \geq 1$  the coefficients of $\delta$ even diminish with $k$. Hence, it is beneficial to have an inexact first-order  oracle of degree $q > 0$, as this allows us to work with less accurate approximation of the (sub)gradient of the nonconvex function $F$ than for $q = 0$.


\subsection{Convex convergence analysis}
\noindent In this section, we analyze the convergence rate of I-PGM for problem \eqref{eq:opt_prblm}, where $F$ is now assumed to be a convex function. By adding extra information to the oracle \eqref{eq:inx2}, we consider the following modification of Definition \ref{eq:def1}:

\begin{definition}\label{eq:def2}
Given a convex function $F$, then it is equipped with an inexact  first-order \red{$(\delta, L)$-oracle} of degree $0\leq q < 2$ if for any $y \in\text{dom} f$ we can compute a vector $ g_{\delta,L,q}(y)$ such that:

\vspace{-0.4cm}

\begin{align}\label{def:inx-q}
0\!\leq\! F(x) \!-\! \left(F(y) \!+\! \langle g_{\delta,L,q}(y),x\!-\!y \rangle\right) \!\leq\! \frac{L}{2}\| x \!-\! y \|^{2} \!\!+\! \delta\| y\!-\!x \|^q \;\;\;  \forall x \!\in\! \text{dom} f. 
\end{align}
\end{definition}

\vspace{-0.4cm}

\noindent Note that Example \ref{eq:ex4} satisfies this definition. In \eqref{def:inx-q}, the zero-order information is considered to be exact. This is not the case in \cite{DeGlNe:13}, which considers the particular choice $q=0$ . Further, the first-order information $g_{\delta,L,q}$ is a subgradient of $f$ at $y$ in  \eqref{def:inx-q}, while in \cite{DeGlNe:13} it is a $\delta$-subgradient. However, using this inexact first-order oracle of degree $q$, I-PGM provides better rates compared to \cite{DeGlNe:13}. From \eqref{def:inx-q} and  \eqref{eq:apd1}, we get:

\vspace{-0.8cm}

\begin{align}\label{def:inx-q2}
0\leq F(x) \!-\! \left( F(y) \!+\! \langle g_{\delta,L,q}(y),x \!-\! y \rangle \right)\leq \frac{L \!+\! q\rho}{2}\norm{x\!-\!y}^{2} + \frac{(2-q)\delta_{q}^{\frac{2}{2-q}}}{2\rho^{\frac{q}{2-q}}},
\end{align} 

\vspace{-0.2cm}

\noindent for all $\rho >0$. Next, we analyze the convergence rate of I-PGM in the convex setting. We have the following convergence rate:

\begin{corollary}\label{th:nonac_q}
Let $F$ be a convex function admitting a \red{$(\delta,L)$-oracle} of degree $q\in[0,2)$ (see Definition \ref{eq:def2}). Let $(x_k)_{k\geq 0}$ be generated by I-PGM  and assume that $\alpha_k = \frac{1}{L+q\rho}$, with $\rho >0$. Define $\hat{x}_k = \frac{\sum_{i=0}^{k}x_{i+1}}{k+1} $ and $R = \| x_0 - x^* \|$. Then, we have:
\begin{align}\label{eq:rate-cxe}
f(\hat{x}_k) - f^* \leq \frac{(L + q\rho)R^2}{2k} + \frac{(2-q)\delta^{\frac{2}{2-q}}}{2\rho^{\frac{q}{2-q}}}. 
\end{align}

\vspace{-0.5cm}

\end{corollary}

\vspace{-0.5cm}

\begin{proof}
Follows from \eqref{def:inx-q2} and Theorem 2 in \cite{DeGlNe:13}.
\end{proof}

\noindent Since we have the freedom of choosing $\rho$, let us minimize the right hand side of \eqref{eq:rate-cxe} over $\rho$. Then,  $\rho$ must satisfy $\frac{qR^2}{2k} - \frac{q \delta^{\frac{2}{2-q}}}{2}\rho^{\frac{-2}{2-q}} = 0.$ Thus, the optimal choice is $\rho = \frac{\delta}{R^{2-q}}k^\frac{2-q}{2}$. Finally, fixing the number of iterations $k$ and replacing this expression in equation \eqref{eq:rate-cxe}, we get:
\begin{align*}
f(\hat{x}_k) - f^* \leq \frac{L R^2}{2k} + \delta\frac{(2+q) R^{q}}{2k^{\frac{q}{2}}}. 
\end{align*}
One can notice that our rate in function values is of order $\mathcal{O}(k^{-1} + \delta k^{-\frac{q}{2}})$, while in \cite{DeGlNe:13} the rate is of order $\mathcal{O}(k^{-1} + \delta)$. Hence, when $q>0$, regardless of the accuracy of the oracle, our second term  diminishes, while in \cite{DeGlNe:13} it remains constant. Hence, our new definition of inexact oracle of degree $q$, Definition \ref{eq:def2}, is also beneficial in the convex case  when analysing proximal gradient type methods, i.e.,  large $q$ yields better rates. 

\medskip 

\noindent \red{We also consider an extension of the fast inexact projected gradient method from \cite{DeGlNe:13}, where the projection is replaced by a proximal step with respect to the function $h$ (see \cite{Nes:13}), called FI-PGM. Note that  the inexactness in FI-PGM comes from the approximate computation of the (sub)gradient of $F$, denoted $g_{\delta_k,L_k,q}(x_k)$, as given in  Definition \ref{eq:def2}. Let $(\theta_k)_{k\geq 0}$ be a sequence such that:
\begin{align}\label{eq:cond_seq}
    \theta_0 \in (0,1],\quad
    \frac{\theta_{k+1}^2}{L_{k+1}}\leq A_{k+1}:=\sum_{i=0}^{k+1}\frac{\theta_i}{L_i} \;\;\;  \forall k\geq 0.
\end{align}
Then, the fast inexact proximal gradient method (FI-PGM) is as follows:
\begin{algorithm}\label{algo2}
\caption{Fast inexact proximal gradient method (FI-PGM)}
\begin{enumerate}
  \item[1.] Given $x_0 \in \text{dom}\; h$, $\theta_0\in (0,1]$ and $0\leq q < 2$.
   \item[] For $k\geq 0$ do:
  \item[2.] Choose $\delta_k$, $L_k$ and 
        $\alpha_k$. Obtain $g_{\delta_k,L_k,q}(x_k)$.
  \item[3.] Compute $y_k = \text{prox}_{\alpha_k
h}\left(x_{k}-\alpha_k g_{\delta_k,L_k,q}(x_k)\right)$.
 \item[4.] Compute $z_k = \argmin \frac12\|x - x_0\|^2 + \sum_{i=0}^{k} \frac{\theta_i}{L_i}\langle  g_{\delta_k,L_k,q}(x_i),x - x_i \rangle + h(x)$.
  \item[5.] Choose $\theta_{k+1}$ satisfying condition \eqref{eq:cond_seq} and compute $A_{k+1} = \sum_{i=0}^{k+1}\frac{\theta_i}{L_i}.$
 \item[6.] Compute $x_{k+1} = \tau_k y_{k} + (1 - \tau_k)z_k$ using $\tau_k = \frac{\theta_{k+1}}{A_{k+1}L_{k+1}}$. 
\end{enumerate}
\end{algorithm}\\
\noindent Using a similar proof  as in \cite{DeGlNe:13}, we get the following convergence rate for FI-PGM algorithm:
\begin{corollary}
Let $F$ satisfy the assumptions of Lemma \ref{th:nonac_q} and  $(y_k)_{k\geq 0}$ be generated by FI-PGM. Then, for all $\rho >0$, we have the following rate:
\begin{align}\label{eq:cve-acc}
f(y_{k}) - f^* \leq \frac{4(L+q\rho)R^2}{(k+1)(k+2)} + \frac{(k+3)(2-q)\delta^{\frac{2}{2-q}}}{2\rho^{\frac{q}{2-q}}}.
\end{align}
\vspace{-0.5cm}
\end{corollary}
\vspace{-0.5cm}
\begin{proof}
The proof follows from \eqref{def:inx-q2} and Theorem 4 in \cite{DeGlNe:13}.
\end{proof}
\noindent The optimal $\rho$ in the right hand side of inequality \eqref{eq:cve-acc} is
\begin{align*}
\rho^* = \frac{\big((k+1)(k+2)(k+3)\big)^{\frac{2-q}{2}}}{(8R^2)^{\frac{2-q}{2}}}\delta.
\end{align*}
Further, replacing $\rho$ with its optimal value in the inequality \eqref{eq:cve-acc}, we get
\begin{align*}
f(y_k) - f^* &\leq \frac{4L R^2}{(k+1)(k+2)} +
\frac{ q8^{\frac{q}{2}} R^q (k+3)}{2((k+1)(k+2)(k+3))^{\frac{q}{2}}} \delta\\
& + \frac{(2-q)8^{\frac{q}{2}} R^q(k+3)}{2((k+1)(k+2)(k+3))^{\frac{q}{2}}} \delta\\
             \quad & =  \frac{4L R^2}{(k+1)(k+2)} + \frac{8^{\frac{q}{2}} R^q(k+3)}{((k+1)(k+2)(k+3))^{\frac{q}{2}}} \delta.\\
             & =  \mathcal{O}\left(\frac{LR^2}{k^2}\right) +  \mathcal{O}\left(\frac{R^q}{k^{\frac{3q}{2} - 1}}\delta\right).   
\end{align*}
Hence, if $q > \frac{2}{3}$, then FI-PGM  doesn't have error accumulation under our inexact oracle as the rate is of order   $\mathcal{O}\left(k^{-2} + \delta k^{1-\frac{3q}{2}}\right)$, while in \cite{DeGlNe:13} the FI-PGM scheme always displays error accumulation, as the convergence rate is of order $\mathcal{O}(k^{-2} + \delta k)$.  Therefore, the same conclusion holds as for I-PGM, i.e., for the FI-PGM  scheme in the convex setting it is beneficial to have an inexact first-order oracle with large degree $q$.} 

\begin{remark}
In our Definition \ref{eq:def1} we have considered exact zero-order information. However, it is possible to change this definition considering also inexact zero-order information for the nonconvex case. More precisely,  we can change Definition \ref{eq:def1} as follows
\begin{equation*}
\left\{\begin{split}
    &F_{\delta_0}(x) - F(x) \leq \delta_0, \\
    &  F(x) \!-\! \left(F_{\delta_0}(y) + \langle g_{\delta,L,q}(y),x-y \rangle\right) \!\leq\! \frac{L}{2}\| x-y \|^{2} + \delta\| x - y \|^q.
    \end{split}\right.
\end{equation*}
With this new definition, the convergence result in Theorem \ref{th:2} becomes:
\begin{align*}
    \sum_{j=0}^{k}\alpha_j \|g_{\delta_{j},L_{j},q}(x_j) + p_{j+1} \|^2\leq f(x_0) - f_\infty + \frac{\sum_{j=0}^{k}(2-q)\delta_j^{\frac{2}{2-q}}}{2\rho^{\frac{q}{2-q}}} + \sum_{j=0}^{k}\delta_0.
\end{align*}
Hence the rate in this case is also influenced by the inexactness of the zero-order information (i.e., $\delta_0$).
Note that for the convex case, the previous extension is not possible in Definition \ref{eq:def2} when $q>0$, since we must have:
\begin{align*}
0\leq  F(x) \!-\! \left(F_{\delta_0}(y) + \langle g_{\delta,L,q}(y),x-y \rangle\right) \!\leq\! \frac{L}{2}\| x-y \|^{2} + \delta\| x - y \|^q,
\end{align*}
which implies for $x=y$ that $F(x) = F_{\delta_0}(x)$. Since we want to have consistency between Definitions \ref{eq:def1} and \ref{eq:def2}, we have chosen to work with the exact zero-order information in our previous nonconvex convergence analysis.
    
\end{remark}


\section{Numerical simulations}
In this section, we evaluate the performance of I-PGM for a composite problem arising in image restoration. Namely, we consider the following nonconvex optimization problem \cite{StThPa:17}:

\vspace{-0.9cm}

\begin{align}\label{eq:prb_sm2}
&\min_{x\in\mathbb{R}^n} \sum_{i=1}^{N} \text{log}\left(\left(a_i^{T}x - b_i\right)^2 + 1\right),\\
&\text{s.t.}\;\; \|x\|_1\leq R,\nonumber
\end{align}

\vspace{-0.2cm}

\noindent where $R>0$, $b\in\mathbb{R}^N$ and $a_i\in\mathbb{R}^n$, for $i=1:N$. In image restoration, $b$ represents the noisy blurred image and $A = (a_1,\cdots,a_N) \in\red{\mathbb{R}^{n\times N}}$ is a blur operator \cite{StThPa:17}. This problem fits into our general problem \eqref{eq:opt_prblm}, with $F(x) = \sum_{i=1}^{N} \text{log}\left( \left(a_i^{T}x - b_i \right)^2 + 1\right)$, which is a nonconvex function with Lipschitz continuous gradient of constant $L_F := \sum_{i=1}^{N}\|a_i\|^2$, and $h(x)$ is the indicator function of the bounded convex set $\lbrace x:\|x\|_1\leq R\rbrace$. 
 We generate the inexact oracle by adding normally distributed random noise $\delta$ to the true gradient, i.e., $g_{\delta,L,q}(x) := \nabla F(x) + \delta$. This is a particular case of Example \ref{eq:ex1}. However, for all $x$ and $y$ satisfying $\|x\| \leq R$, $\|y\| \leq R $, we have the following:
\begin{align*}
    \delta\|x - y\|&= \delta\|x - y\|^{1-q}\|x - y\|^q \leq \delta(2R)^{1-q}\|x - y\|^q. 
\end{align*}
Thus, this example  satisfies Definition \ref{eq:def1} for all $q\in  [0, 1]$. We apply I-PGM for this particular example where we consider three choices for the degree $q$:  $0$, $1/2$ and $1$. Recall that the convergence rate of I-PGM with constant step size is (see Corollary \ref{eq:col1}, first statement):
\begin{align}\label{eq:uper_bnd}
&\min_{j=0:k} \norm{g_j \!+\! p_{j+1}}^2 \leq \frac{2(q\!+\!1)L(f(x_0) \!-\! f^*)}{k+1} +(q\!+\!1)(2\!-\!q)L^{\frac{2-2q}{2-q}}\delta^{\frac{2}{2-q}}.
\end{align} 
At each iteration of I-PGM we need to solve the following convex subproblem:
\begin{align*}
&\min_{x\in\mathbb{R}^n}  F(x_k) + \langle g_{\delta,q}(x_k),x-x_k\rangle + \frac{L + q\rho}{2}\|x - x_k\|^2,\;\; \text{s.t.} \;\; \|x\|_1 \leq R.
\end{align*}
This subproblem has a closed form solution (see e.g.,  \cite{Tib:96}). We compare I-PGM with constants step size $\alpha_k = \frac{1}{2(L_F+q\rho)}$ and $\rho = L_F$ for three  choices of $q=0, 1/2, 1$ and three choices of noise norm $\|\delta\| \leq 0.1, 1, 3$, respectively. The results are given in Fig. 1 (dotted lines), \red{where we plot the evolution of the error $\min_{j=0:k}\|\frac{1}{\alpha_k}(x_{j+1} - x_j)\|^2$, which corresponds to the gradient mapping}. In the same figure  we also plot the theoretical bounds \eqref{eq:uper_bnd} for $q=0, 1/2, 1$ (full lines). { Our main figures are Figure \ref{fig:three graphs}(a), (c), and (d), while Figure \ref{fig:three graphs}(b) is a subfigure (zoom) of Figure \ref{fig:three graphs}(a), displaying only the first 300 iterations. Moreover, one can see in these main figures (i.e., Figure \ref{fig:three graphs}(a), (c), and (d)) that the behaviour of our algorithm for $q=1$ is better than for $q=1/2$. Similarly, the behaviour of our algorithm for $q=1/2$ is better than for $q=0$. One can observe these better behaviours after 300 iterations when the error $\delta$ is small (see Figure \ref{fig:three graphs}(c) and (d)). However, when the error $\delta$ is large, we need to perform a larger number of iterations before we can observe these behaviours,  (see Figure \ref{fig:three graphs}(a) and (b)). This is  natural, since large errors on the gradient approximation  must have impact on the convergence speed.} Hence, as the degree $q$ increases or the norm of the noise decreases, better accuracies for the norm of the gradient mapping can be achieved, which supports our theoretical findings. 

\medskip

\noindent \red{ Moreover, from the numerical simulations, one can observe that the gap between the theoretical and the practical bounds is large in Figure \ref{fig:three graphs}(c) and (d). We believe  that this happens because, in the convergence analysis, the theoretical bounds are derived under worst-case scenarios (i.e., the convergence analysis must account for the worst case direction generated by the inexact first-order  oracle, while in practical implementations, which often involve randomness, one usually doesn't encounter these worst-case directions). However, the simulations in Figure \ref{fig:three graphs}(a) show that the gap between the theoretical bounds and the practical behavior is not too large. More precisely, we have generated at each iteration 100 random directions and, in order to update the new point, we have chosen the worst direction with respect to the gradient mapping (i.e., the largest)  $\|x_{k+1} - x_k\|$). The results are given in Figure \ref{fig:three graphs}(a), where one can see that the theoretical and practical bounds are getting closer for sufficiently large number of iterations.}


\begin{figure}[ht]
     \centering
     \begin{subfigure}[b]{0.45\textwidth}
         \centering
         \includegraphics[width=\textwidth]{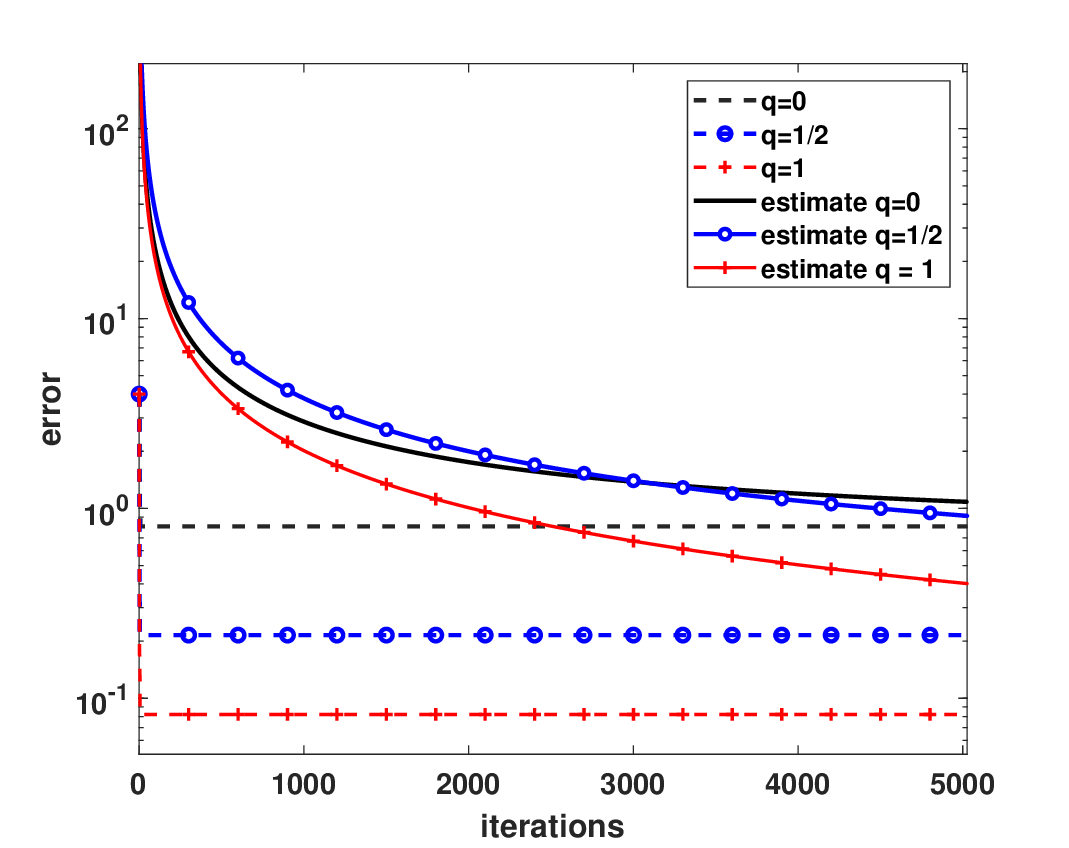}
         \caption{$\|\delta\| \leq 3$}
         \label{fig:1}
     \end{subfigure}
     \begin{subfigure}[b]{0.45\textwidth}
         \centering
         \includegraphics[width=\textwidth]{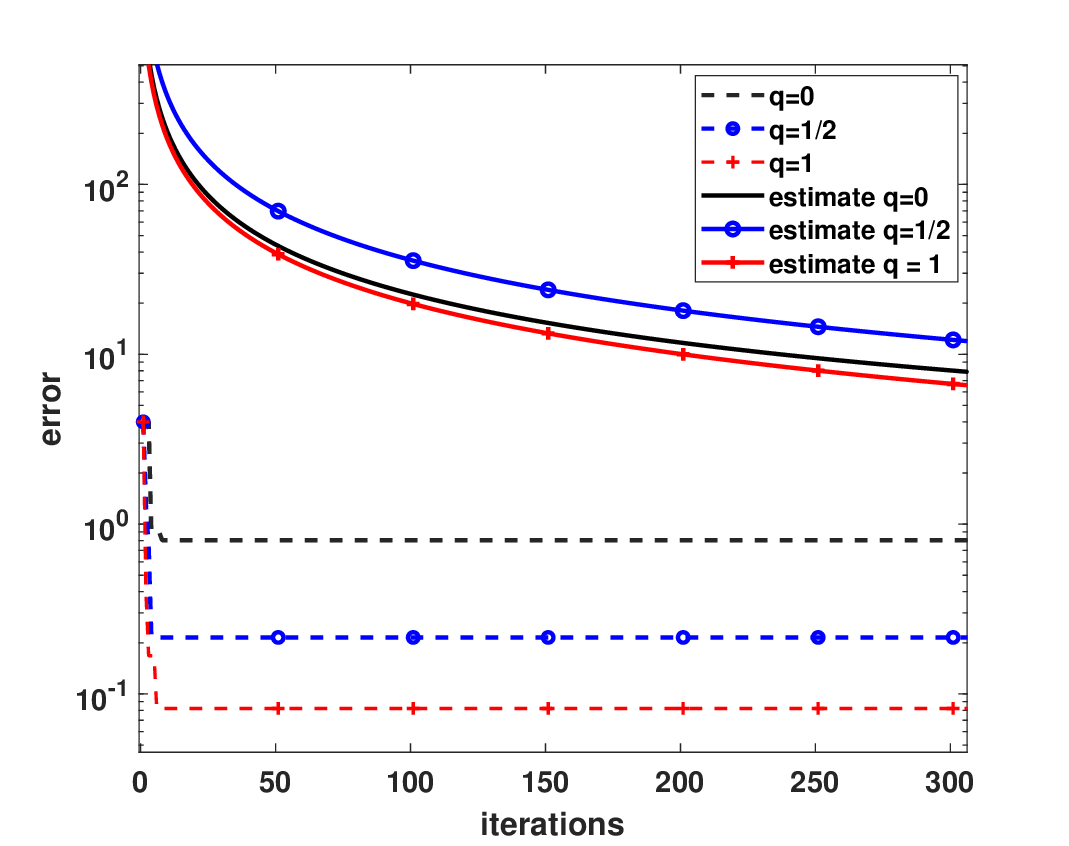}
         \caption{$\|\delta\| \leq  3$}
         \label{fig:2}
     \end{subfigure}
     \hspace{0.4cm}
     \begin{subfigure}[b]{0.45\textwidth}
         \centering
         \includegraphics[width=\textwidth]{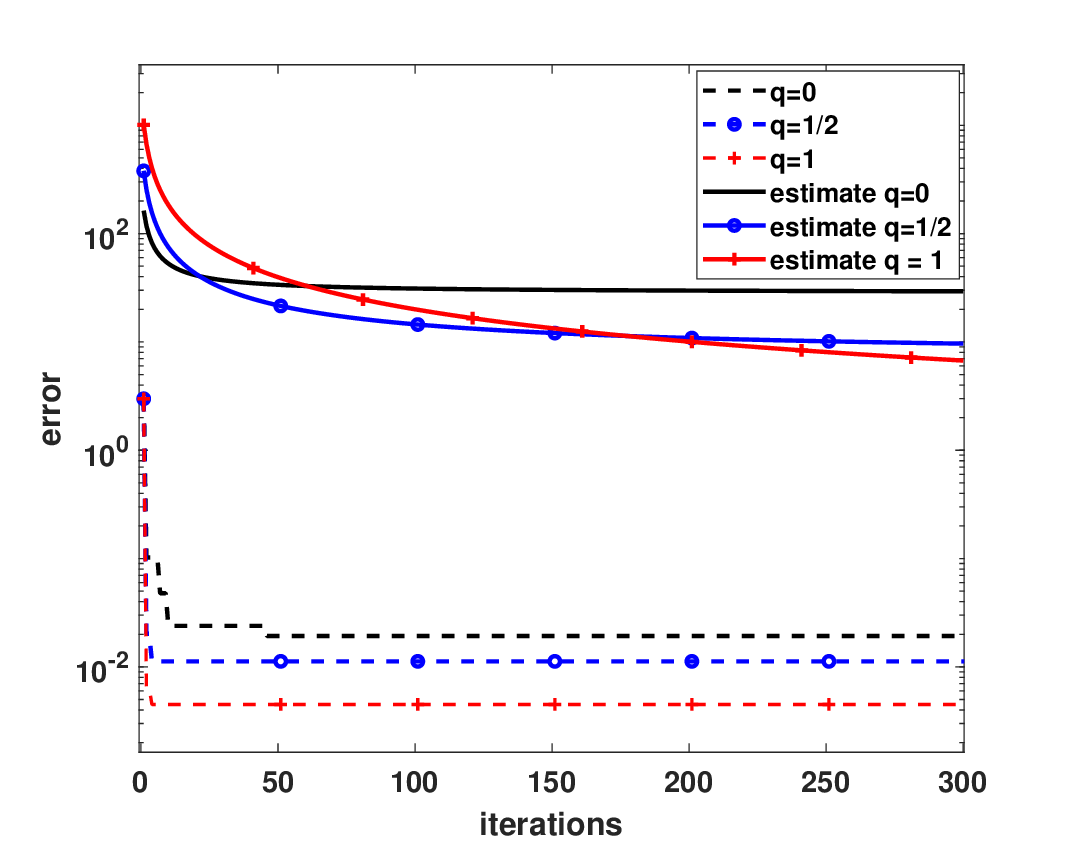}
         \caption{$\|\delta\| \leq 1$}
         \label{fig:3}
     \end{subfigure}
          \hspace{0.4cm}
     \begin{subfigure}[b]{0.45\textwidth}
         \centering
         \includegraphics[width=\textwidth]{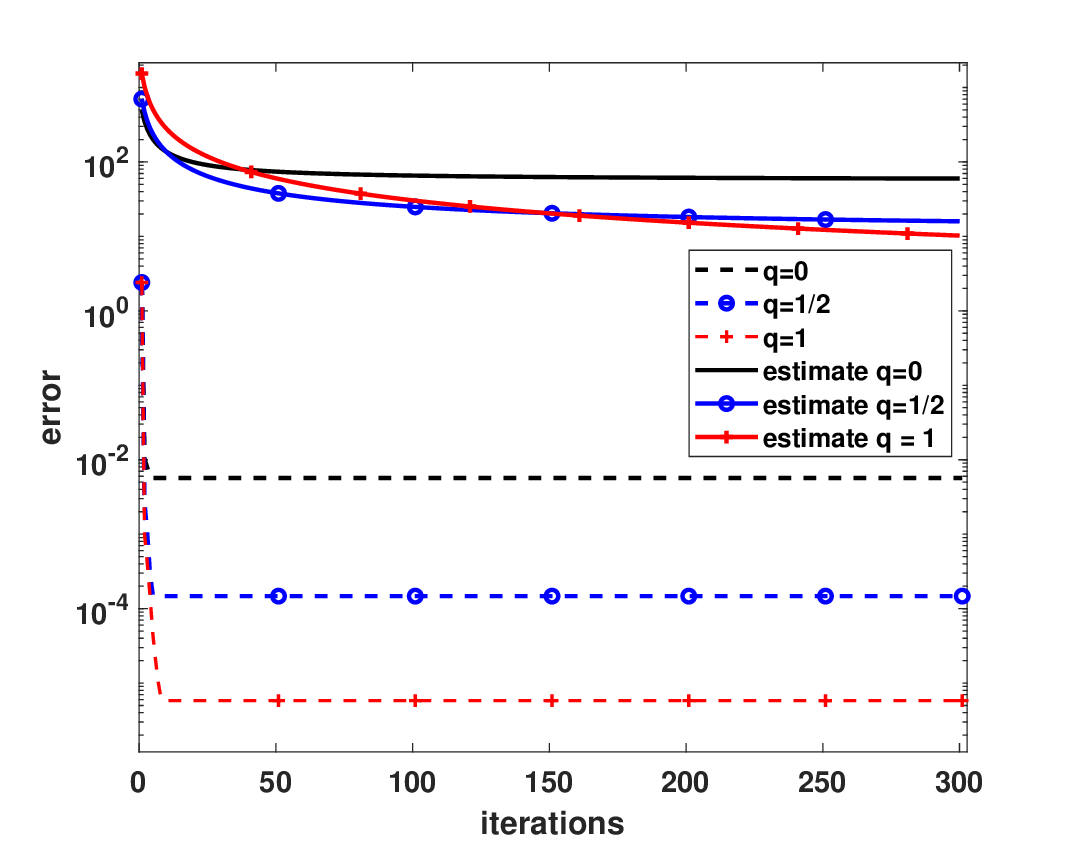}
         \caption{$\|\delta\| \leq 0.1$}
         \label{fig:4}
     \end{subfigure}
        \caption{Practical (dotted lines) and theoretical (full lines) performances of the I-PGM algorithm for different choices of $q$ and $\delta$, and with  $R = 4$. {Figure (b) represents a zoom of the left corner from Figure (a).} }
        \label{fig:three graphs}
\end{figure}

\bmhead{Acknowledgments}
\noindent The research leading to these results has received funding from: ITN-ETN project TraDE-OPT funded by the European Union’s Horizon 2020 Research and Innovation Programme under the Marie Skolodowska-Curie grant agreement No. 861137;  UEFISCDI PN-III-P4-PCE-2021-0720, under project L2O-MOC, nr. 70/2022.\\

\medskip 

\noindent Data sharing not applicable to this article as no datasets were generated or analyzed during the current study.


\begin{thebibliography}{}
\bibitem{AgKam:17}
A. Agafonov, D. Kamzolov, P. Dvurechensky, A. Gasnikov,  M. Takac,  \emph{Inexact Tensor Methods and Their Application to Stochastic Convex Optimization}, Optimization Methods and Software, doi.org/10.1080/10556788.2023.2261604, 2017.

\bibitem{Bec:17}
A. Beck, \emph{First-order methods in optimization}, vol. 25, SIAM, 2017.

\bibitem{BoGuRa:16}
L. Bogolubsky, G. Gusev, A. Raigorodskii, A. Tikhonov, M. Zhukovskii, P. Dvurechensky, A. Gasnikov, Yu. Nesterov, \emph{ Learning supervised PageRank with gradient-based and gradient-free optimization methods}, 30th Conference on Neural Information Processing Systems, 2016.


\bibitem{Bot:10}
L.  Bottou, \emph{Large-scale machine learning with stochastic gradient descent}, 19th International Conference on Computational Statistics, 2010.


\bibitem{CoDiOr:18}
M.B. Cohen, J. Diakonikolas, L. Orecchia, \emph{On Acceleration with Noise-Corrupted Gradients}, International Conference on Machine Learning, 2018. 


\bibitem{Asp:08}
A. d’Aspremont,  \emph{ Smooth optimization with approximate gradient}, SIAM Journal on Optimization, 19(3): 1171--1183, 2008.


\bibitem{DeGlNe:13}
O. Devolder, F. Glineur, Yu. Nesterov, \emph{First-order methods of smooth convex optimization
with inexact oracle}, Math. Prog., 146: 37--75, 2013.


\bibitem{DvGa:16}
P. Dvurechensky,  A. Gasnikov, \emph{Stochastic intermediate gradient method for convex problems with stochastic inexact oracle}, Journal of Optimization Theory and Applications, 171(1): 121--145, 2016.


\bibitem{Dvur:17}
P. Dvurechensky,	\emph{ A gradient method with inexact oracle for composite nonconvex optimization}, Computer Research and Modeling, 14(2), 2022.    


\bibitem{GeBe:99}
{A. Geary and D.P. Bertsekas,  \emph{Incremental subgradient methods for nondifferentiable optimization}, Conference on Decision and Control, 1999.}


\bibitem{Hir:79}
J-B. Hiriart-Urruty, \emph{New concepts in nondifferentiable programming},  Memoires de la Societe Mathematique de France,  60:  57--85, 1979.



\bibitem{StThPa:17}
L. Stella, A. Themelis and P. Patrinos, \emph{Forward–backward quasi-Newton methods for nonsmooth optimization problems}, Computational Optimization and Applications, 67: 443--487, 2017.



\bibitem{Mor:06}
\red{B. Mordukhovich, \emph{Variational analysis and generalized differentiation: basic theory}, Springer, 2006.}


\bibitem{Nes:04}
Yu.  Nesterov, \emph{Introductory lectures on convex optimization: A basic course},  Springer, 2004.



\bibitem{Nes:13}
Yu.  Nesterov, \emph{Gradient methods for minimizing composite functions}
Mathematical programming 140.1, 125--161, 2013.



\bibitem{Pol:63}
B.T. Polyak, \emph{Gradient methods for the minimisation of functionals}, Computational Mathematics and Mathematical Physics, \!3(4): \!864--878, \!1963.




\bibitem{RoWe:98}
\red{R. Rockafellar and R. Wets, \emph{Variational Analysis}, Springer, 1998.}


\bibitem{StDv:19}
F.S. Stonyakin, D. Dvinskikh, P. Dvurechensky, A. Kroshnin, O. Kuznetsova, A. Agafonov, A. Gasnikov, A. Tyurin, C.A. Uribe, D. Pasechnyuk, S. Artamonov,  \emph{Gradient methods for problems with inexact model of the objective}, International Conference on Mathematical Optimization Theory and Operations Research, 2019. 



\bibitem{Tib:96}
R. Tibshirani, \emph{Regression shrinkage and selection via the Lasso}, Journal of the Royal Statistical Society, 58(1): 267--288, 1996.

\bibitem{WaLI:14}
P.W. Wang and C-J. Lin, \emph{Iteration complexity of feasible descent methods for convex optimization}, Journal of Machine Learning Research,  15(4), 1523--1548, 2014.
\end{thebibliography}
\end{document}